\documentclass{article}
\usepackage{amsmath,bm,colortbl,graphicx,mathrsfs,afterpage,amssymb,url,cite,verbatim,amsthm}
\RequirePackage[colorlinks,citecolor=blue,urlcolor=blue]{hyperref}
\pdfoutput=1
\theoremstyle{plain}
\newtheorem{theorem}{Theorem}

\newtheorem{corollary}[theorem]{Corollary}
\theoremstyle{remark}
\newtheorem*{remark}{Remark}
\def\E{{\rm E}} % Expectation
\def\Var{{\rm Var}} % Variance
\def\Cov{{\rm Cov}} % Variance
\def\eps{\varepsilon}
\newcommand{\seqnum}[1]{\href{http://oeis.org/#1}{#1}}
\allowdisplaybreaks
\title{Parrondo games with two-dimensional \\ spatial dependence}
\author{S. N. Ethier\thanks{Department of Mathematics, University of Utah, 155 South 1400 East, Salt Lake City, UT 84112 USA. \href{mailto:ethier@math.utah.edu}{ethier@math.utah.edu}.  Partially supported by a grant from the Simons Foundation (209632).} \ and Jiyeon Lee\thanks{Department of Statistics, Yeungnam University, 214-1 Daedong, Kyeongsan, Kyeongbuk 712-749, South Korea. \href{mailto:leejy@yu.ac.kr}{leejy@yu.ac.kr}.  Supported by a 2014 Yeungnam University Research Grant.}}
\date{} 
\begin{document}
\maketitle

\begin{abstract}
Parrondo games with one-dimensional spatial dependence were introduced by Toral and extended to the two-dimensional setting by Mihailovi\'c and Rajkovi\'c.  $MN$ players are arranged in an $M\times N$ array.  There are three games, the fair, spatially independent game $A$, the spatially dependent game $B$, and game $C$, which is a random mixture or nonrandom pattern of games $A$ and $B$.  Of interest is $\mu_B$ (or $\mu_C$), the mean profit per turn at equilibrium to the set of $MN$ players playing game $B$ (or game $C$).  Game $A$ is fair, so if $\mu_B\le0$ and $\mu_C>0$, then we say the \textit{Parrondo effect} is present.    

We obtain a strong law of large numbers and a central limit theorem for the sequence of profits of the set of $MN$ players playing game $B$ (or game $C$).  The mean and variance parameters are computable for small arrays and can be simulated otherwise.  The SLLN justifies the use of simulation to estimate the mean.  The CLT permits evaluation of the standard error of a simulated estimate.  We investigate the presence of the Parrondo effect for both small arrays and large ones.  One of the findings of Mihailovi\'c and Rajkovi\'c was that ``capital evolution depends to a large degree on the lattice size.''  We provide evidence that this conclusion is incorrect.  Part of the evidence is that, under certain conditions, the means $\mu_B$ and $\mu_C$ converge as $M,N\to\infty$.  Proof requires that a related spin system on ${\bf Z}^2$ be ergodic.  However, our sufficient conditions for ergodicity are rather restrictive.
\end{abstract}

\section{Introduction}\label{intro}

Parrondo games with one-dimensional spatial dependence were introduced by Toral \cite{T01} and extended to the two-dimensional setting by Mihailovi\'c and Raj\-kovi\'c \cite{MR06}.  The basic game depends on two integer parameters, $M\ge3$ and $N\ge3$, and five probability parameters, $p_0$, $p_1$, $p_2$, $p_3$, $p_4\in[0,1]$.  There are $MN$ players arranged in an $M\times N$ array with periodic boundary conditions.  At each turn, one player is chosen at random to play.  Suppose it is the player at $(i,j)$, with $1\le i\le M$ and $1\le j\le N$.   He tosses what we call ``coin $m$'' with probability $p_m$ of heads if $m$ is the number of his nearest neighbors (i.e., the players at $(i+1,j)$, $(i-1,j)$, $(i,j+1)$, and $(i,j-1)$) who are winners.  (A winner is a player whose last game resulted in a win.)  Because of the periodic boundary conditions, $(M+1,j):=(1,j)$, $(0,j):=(M,j)$, $(i,N+1):=(i,1)$, and $(i,0):=(i,N)$.  The player wins one unit with heads and loses one unit with tails.  The game can be initialized arbitrarily.  

The game just described is known as game $B$.  Game $A$ is the special case of game $B$ in which $p_0=p_1=p_2=p_3=p_4=1/2$; in particular, it is not spatially dependent.  Game $C$ is either $(i)$ a $(\gamma,1-\gamma)$ random mixture of game $A$ and game $B$ (toss a coin with probability $\gamma$ of heads and play game $A$ if heads appears, game $B$ if tails), where $0<\gamma<1$, usually denoted by $C:=\gamma A+(1-\gamma)B$, or $(ii)$ a nonrandom pattern of games $A$ and $B$, played repeatedly; we restrict attention to patterns in which $r$ plays of game $A$ are followed by $s$ plays of game $B$, where $r$ and $s$ are positive integers, usually denoted by $C:=A^r B^s$.  

Of interest is $\mu_B$ (or $\mu_C$), the mean profit per turn at equilibrium to the set of $MN$ players playing game $B$ (or game $C$).  Game $A$ is fair, so if $\mu_B\le0$ and $\mu_C>0$, then we say the \textit{Parrondo effect} is present, or that $\bm p=(p_0,p_1,p_2,p_3,p_4)$ belongs to the \textit{Parrondo region}.  This means that the games with parameter vector $\bm p$ provide an example of \textit{Parrondo's paradox}, in which two fair or losing games ($A$ and $B$) combine to form a winning game ($C$).  Similarly, if $\mu_B\ge0$ and $\mu_C<0$, then we say the \textit{anti-Parrondo effect} is present, or that $\bm p$ belongs to the \textit{anti-Parrondo region}.    

We obtain a strong law of large numbers (SLLN) and a central limit theorem (CLT) for the sequence of profits of the set of $MN$ players playing game $B$ (or playing game $C$).  The mean and variance parameters are computable for $MN\le20$ (at least) and can be simulated otherwise.  The SLLN justifies the use of simulation to estimate the mean.  The CLT permits evaluation of the standard error of a simulated estimate.  To get a sense of what the Parrondo region looks like when $M=N=3$, we fix $p_0$ and $p_4$ and graph the surfaces $\mu_B=0$ and $\mu_C=0$ in the $(p_1,p_3,p_2)$ unit cube.  The region below the former and above the latter is a three-dimensional cross-section of the Parrondo region.

Actually, the phrase ``below the former'' requires some clarification because, paradoxically, $\mu_B$, as a function of $p_2$ for fixed $p_0$, $p_1$, $p_3$, and $p_4$, is not necessarily increasing (or even nondecreasing).  That is, it is possible that increasing $p_2$, the favorability of coin 2, can decrease $\mu_B$, the mean profit from game $B$.

One of the main conclusions of Mihailovi\'c and Rajkovi\'c \cite{MR06} was that ``capital evolution depends to a large degree on the lattice (i.e., array) size.''  We provide evidence that this conclusion is incorrect, just as it is known to be in the case of the one-dimensional spatial model (Ref.~\cite{EL12a}).  Moreover, we believe that Mihailovi\'c and Rajkovi\'c were misled by simulations whose sample size was inadequate.  Part of the evidence is that, under certain conditions, the means $\mu_B$ and $\mu_C$ converge as $M,N\to\infty$.  By analogy with the one-dimensional case, the proof requires that a related spin system on ${\bf Z}^2$ be ergodic.  Unlike in the one-dimensional case, however, our sufficient conditions for ergodicity are rather restrictive.

Parrondo games were first introduced as a discretized version of the flashing Brownian ratchet (Ajdari and Prost \cite{AP92}), in which an asymmetric ratchet potential is switched on (game $B$) and off (game $A$) repeatedly, effecting motion (game $C:=AB$).  The first examples were one-player games in which the asymmetric game $B$ was capital dependent (Harmer and Abbott \cite{HA99}) or history dependent (Parrondo, Harmer, and Abbott \cite{PHA00}).  Multi-player Parrondo games were introduced by Din\'is and Parrondo \cite{DP03}, Toral \cite{T01,T02}, and others.  Of these, perhaps Toral's \cite{T01} spatially dependent games have generated the most interest.  He referred to them as \textit{cooperative} Parrondo games but we prefer the term \textit{spatially dependent} Parrondo games so as to avoid conflict with the apparently unrelated field of cooperative game theory.  Most works on these games (Refs.~\cite{T01,MR03,X11,EL12a,EL12b,CL12,EL13a,EL13b,L14,Lee14,Y15,EL15b}) have been concerned with the one-dimensional model, and until now only Mihailovi\'c and Rajkovi\'c \cite{MR06} have discussed the two-dimensional model.  

See Abbott \cite{A10} for the most recent review of Parrondo's paradox.

\section{The Markov chain and its reduction}\label{MC}

A Markov chain can be defined that keeps track of the status (loser or winner, 0 or 1) of each of the $MN$ players playing game $B$, where $M,N\ge3$. Its state space is the set of $M\times N$ arrays of 0s and 1s, that is, the product space 
\begin{equation*}
\Sigma:=\{\bm x=(x_{i,j}): x_{i,j}\in\{0,1\}{\rm\ for\ }i=1,\ldots,M\;\text{and}\;j=1,\ldots,N\}=\{0,1\}^{MN}
\end{equation*}
with $2^{MN}$ states.  With the help of some notation, we can specify the one-step transition matrix.  Let $m_{i,j}(\bm x):=x_{i+1,j}+x_{i-1,j}+x_{i,j+1}+x_{i,j-1}$ be the number (0, 1, 2, 3, or 4) of winners among the four nearest neighbors of the player at $(i,j)$.  Of course, first subscripts $M+1$ and 0 are 1 and $M$; second subscripts $N+1$ and 0 are 1 and $N$.  Also, let $\bm x^{i,j}$ be the element of $\Sigma$ equal to $\bm x$ except at entry $(i,j)$.

The one-step transition matrix $\bm P$ for this Markov chain depends on $M\ge3$ and $N\ge3$ and on the five probability parameters, $p_0,p_1,p_2,p_3,p_4\in[0,1]$.  It has the form, for each $\bm x\in\Sigma$,
\begin{equation}\label{x to x^ij}
P(\bm x,\bm x^{i,j}):=\begin{cases}(MN)^{-1}p_{m_{i,j}(\bm x)}&\text{if $x_{i,j}=0$,}\\(MN)^{-1}q_{m_{i,j}(\bm x)}&\text{if $x_{i,j}=1$,}\end{cases}\quad i=1,\ldots,M,\; j=1,\ldots,N,
\end{equation}
\begin{equation}\label{x to x}
P(\bm x,\bm x):=(MN)^{-1}\bigg(\sum_{i,j:x_{i,j}=0}q_{m_{i,j}(\bm x)}+\sum_{i,j:x_{i,j}=1}p_{m_{i,j}(\bm x)}\bigg),
\end{equation}
and $P(\bm x,\bm y)=0$ otherwise, where $q_m:=1-p_m$ for $m=0,1,2,3,4$ and empty sums are 0.  We assume for now that $0<p_m<1$ for $m=0,1,2,3,4$, in which case the Markov chain is irreducible and aperiodic; we weaken this assumption in Section \ref{reducible}.  

For example, if $M=N=3$ and
$$
\arraycolsep=1.mm
\bm x:=\begin{pmatrix}0&0&1\\0&1&0\\1&0&1\end{pmatrix},\qquad\bm x':=\begin{pmatrix}0&0&1\\0&1&1\\1&1&1\end{pmatrix},
$$
we have $P(\bm x,\bm x')=0$, contrary to Mihailovi\'c and Rajkovi\'c \cite{MR06}, because at most one entry of $\bm x$ can change in one time step.  However, $P^2(\bm x,\bm x')=2[(1/9)p_3]^2$. 

The description of the model suggests that its long-term behavior should be invariant under rotation and/or reflection of rows and/or columns of the $M\times N$ array of players, as well as under matrix transposition if $M=N$.  In order to maximize the values of $M$ and $N$ for which exact computations are feasible, we use this idea to effectively reduce the size of the state space.  The technical details are explained in Lemma 1 of Ref.~\cite{EL12a}.

The lemma applies to our Markov chain with $MN$ playing the role of $N$ and $G$ being the subgroup of permutations of 
$$
\arraycolsep=3mm
\begin{pmatrix}1&2&\cdots&N\\N+1&N+2&\cdots&2N\\\vdots&\vdots&&\vdots\\(M-1)N+1&(M-1)N+2&\cdots&MN\end{pmatrix}
$$
(written for convenience as a matrix instead of as the $MN$-dimensional vector $(1,2,\ldots,MN)$) generated by $\sigma_1$, $\sigma_2$, $\sigma_3$, and $\sigma_4$, where $\sigma_1$ rotates the rows (row 1 becomes row 2, row 2 becomes row 3, and so on, and row $M$ becomes row 1), $\sigma_2$ reflects the rows (the rows are reverse ordered: row 1 and row $M$ are interchanged, row 2 and row $M-1$ are interchanged, and so on), $\sigma_3$ rotates the columns, and $\sigma_4$ reflects the columns.  In the case of a square array (i.e., $M=N$), we can include, along with these four permutations, $\sigma_5$, which transposes the matrix.

We must check that the condition 
\begin{equation}\label{lumpability}
P(\bm x_\sigma,\bm y_\sigma)=P(\bm x,\bm y),\qquad \bm x, \bm y\in\Sigma, 
\end{equation}
(from Lemma~1 of Ref.~\cite{EL12a}), where $(\bm x_\sigma)_{i,j}:=x_{\sigma(i,j)}$, is satisfied by these five permutations, and for this it is enough to verify that
$m_{i,j}(\bm x_\sigma)=m_{\sigma(i,j)}(\bm x)$ for $i=1,\ldots,M$, $j=1,\ldots,N$, and all $\bm x\in\Sigma$, whenever $\sigma$ is given by $\sigma_1$, $\sigma_2$, $\sigma_3$, $\sigma_4$, or (if $M=N$) $\sigma_5$.  

The practical effect of this is that we can reduce the size of the state space (namely, $2^{MN}$) to what we call its \textit{effective size}, which is simply the number of equivalence classes.  For example, if $M=N=3$ and we use $\sigma_1$, $\sigma_2$, $\sigma_3$, and $\sigma_4$, there are $2^9=512$ states and 36 equivalence classes; see Ref.~\cite{E13} for a list.  If we use $\sigma_5$ as well, there are only 26 equivalence classes; see Ref.~\cite{EL15} for a list.  Table \ref{sequences} lists the number of equivalence classes for small arrays.

\begin{table}[ht]
\caption{\label{sequences}The size and effective size of the state space for an $M\times N$ array ($3\le M\le N$ and $MN\le25$).  Columns 3 and 4 are from Sloan \cite{S15}, specifically, \seqnum{A222188} and \seqnum{A255016}.\medskip}
\catcode`@=\active \def@{\hphantom{0}}
\begin{center}
\begin{footnotesize}
\begin{tabular}{ccccc}
\noalign{\smallskip}
\hline
\noalign{\smallskip}
number of   &  size of   &  effective size  & effective size \\
players & @@@state space@@@ & with row/column  & with, in addition, \\
 $M\times N$ & $2^{MN}$ &  rotation/reflection &  transposition,  \\
      &          &                      &    if $M=N$ \\
\noalign{\smallskip}
\hline
\noalign{\smallskip}
$@3\times3$ & @@@@@512 & @@@@36 &     @@@@26 \\
$@3\times4$ & @@@@4096 & @@@158\\  
$@3\times5$ & @@@32768 & @@@708\\  
$@3\times6$ & @@262144 & @@4236\\  
$@3\times7$ & @2097152 & @26412\\ 
$@3\times8$ & 16777216 & 180070\\ 
\noalign{\medskip}
$@4\times4$ & @@@65536 & @@1459 &     @@@805 \\
$@4\times5$ & @1048576 & @14676\\
$@4\times6$ & 16777216 & 184854\\
\noalign{\medskip}
$@5\times5$ & 33554432 & 340880 &     172112 \\
\noalign{\smallskip}
\hline
\end{tabular}
\end{footnotesize}
\end{center}
\end{table}

The reduced Markov chain has state space $\bar\Sigma$ (the set of equivalence classes) and one-step transition matrix $\bar{\bm P}$ given by
\begin{equation}\label{Pbar}
\bar P([\bm x],[\bm y]):=\sum_{{\bm y}'\in[\bm y]}P(\bm x,{\bm y}'),
\end{equation}
where $[\bm x]\in\bar\Sigma$ denotes the equivalence class containing $\bm x\in\Sigma$.  For \eqref{Pbar} to be well defined, $\bm P$ must be \textit{lumpable} with respect to the equivalence relation (Kemeny and Snell \cite[p.~124]{KS76}), and \eqref{lumpability} is a sufficient condition for this.

\section{SLLN and CLT}\label{SLLN}

We need versions of the SLLN and the CLT that are suited to game $B$ and game $C:=\gamma A+(1-\gamma)B$.  The key result is Theorem 1 of Ref.~\cite{EL09}.  

Our original Markov chain has state space $\Sigma:=\{0,1\}^{MN}$ and its one-step transition matrix $\bm P$ is given by (\ref{x to x^ij}) and (\ref{x to x}), where $q_m:=1-p_m$ and we assume for now that $0<p_m<1$ for $m=0,1,2,3,4$.  The Markov chain is irreducible and aperiodic.  Theorem 1 of Ref.~\cite{EL09} does not apply directly because the payoffs are not completely specified by the one-step transitions of the Markov chain.  For example, unless $\bm x$ has all 0s or all 1s, a transition from $\bm x$ to $\bm x$ could be the result of a win or a loss.  One way around this is to augment the state space.  In Ref.~\cite{EL12a} we kept track not only of $\bm x\in\Sigma$ but also of the label of the next player to play.  Here a different augmentation is more effective.  We let $\Sigma^\circ:=\Sigma\times \{-1,1\}$ and keep track not only of $\bm x\in\Sigma$ but also of the profit from the last game played, say $s\in\{-1,1\}$.  The new one-step transition matrix $\bm P^\circ$ has the form, for every $(\bm x,s)\in\Sigma^\circ$,
\begin{equation}\label{x to x^ij,1}
P^\circ((\bm x,s),(\bm x^{i,j},1)):=\begin{cases}(MN)^{-1}p_{m_{i,j}(\bm x)}&\text{if $x_{i,j}=0$,}\\
0&\text{if $x_{i,j}=1$,}\end{cases}
\end{equation}
\begin{equation}\label{x to x^ij,-1}
P^\circ((\bm x,s),(\bm x^{i,j},-1)):=\begin{cases}0&\text{if $x_{i,j}=0$,}\\
(MN)^{-1}q_{m_{i,j}(\bm x)}&\text{if $x_{i,j}=1$,}\end{cases}
\end{equation}
for $i=1,\ldots,M$ and $j=1,\ldots,N$, and
\begin{equation}\label{x to x,1}
P^\circ((\bm x,s),(\bm x,1)):=(MN)^{-1}\sum_{i,j:x_{i,j}=1}p_{m_{i,j}(\bm x)},
\end{equation}
\begin{equation}\label{x to x,-1}
P^\circ((\bm x,s),(\bm x,-1)):=(MN)^{-1}\sum_{i,j:x_{i,j}=0}q_{m_{i,j}(\bm x)},
\end{equation}
where $q_m:=1-p_m$ for $m=0,1,2,3,4$ and empty sums are 0.  There are two inaccessible states, $(\bm 0,1)$ and $(\bm 1,-1)$, but if they were excluded from the state space, the Markov chain would be irreducible and aperiodic.  It will be convenient to keep them, so we let $\bm\pi^\circ$ denote the unique stationary distribution, which has entry 0 at each of the two inaccessible states.  The payoff function $w^\circ$ can now be defined by
\begin{eqnarray*}
w^\circ((\bm x,s),(\bm x^{i,j},t))=t\text{ if $x_{i,j}=(1-t)/2$,}\qquad w^\circ((\bm x,s),(\bm x,t))=t
\end{eqnarray*}
for all $(\bm x,s)\in\Sigma^\circ$, $i=1,2,\ldots,M$, $j=1,2,\ldots,N$, and $t\in\{-1,1\}$, and $w^\circ=0$ otherwise.  This allows us to define the matrix $\bm W^\circ$ and then $\dot{\bm P}^\circ:={\bm P}^\circ\circ \bm W^\circ$ and $\ddot{\bm P}^\circ:={\bm P}^\circ\circ\bm W^\circ\circ\bm W^\circ$, the Hadamard (or entrywise) products.  Theorem~1 of Ref.~\cite{EL09} yields the following.

\begin{theorem}\label{SLLN/CLT}
Let $0<p_m<1$ for $m=0,1,2,3,4$, so that the Markov chain with one-step transition matrix $\bm P^\circ$ is ergodic, and let the row vector $\bm\pi^\circ$ be its unique stationary distribution.  Define
\begin{equation*}
\mu^\circ=\bm\pi^\circ\dot{\bm P}^\circ\bm1,\qquad (\sigma^\circ)^2=\bm\pi^\circ\ddot{\bm P}^\circ\bm 1-(\bm\pi^\circ\dot{\bm P}^\circ\bm 1)^2+2\bm\pi^\circ\dot{\bm P}^\circ(\bm Z^\circ-\bm1\bm\pi^\circ)\dot{\bm P}^\circ\bm 1.
\end{equation*}
where $\bm1$ denotes a column vector of $1$s with entries indexed by $\Sigma^\circ$ and $\bm Z^\circ:=(\bm I-(\bm P^\circ-\bm1\bm\pi^\circ))^{-1}$ is the fundamental matrix.  (Notice that $\bm1\bm\pi^\circ$ is the square matrix each of whose rows is equal to $\bm\pi^\circ$.)  Let $\{X_n^\circ\}_{n\ge0}$ be a time-homogeneous Markov chain in $\Sigma^\circ$ with one-step transition matrix $\bm P^\circ$, and let the initial distribution be arbitrary.  For each $n\ge1$, define $\xi_n:=w^\circ(X_{n-1}^\circ,X_n^\circ)$ and $S_n:=\xi_1+\cdots+\xi_n$.  Then $\lim_{n\to\infty}n^{-1}S_n=\mu^\circ$ \emph{a.s.} and, if $(\sigma^\circ)^2>0$, then $(S_n-n\mu^\circ)/\sqrt{n(\sigma^\circ)^2}\to_d N(0,1)$ as $n\to\infty$.
\end{theorem}

\begin{remark}
The abbreviation ``a.s.'' stands for ``almost surely,'' meaning ``with probability 1.''  The symbol $\to_d$ denotes convergence in distribution.
\end{remark}

We next show that there are simpler expressions for this mean and variance.  Let us define
\begin{eqnarray*}
\mu:=\bm\pi\dot{\bm P}\bm1,\qquad\sigma^2&:=&\bm\pi\ddot{\bm P}\bm 1-(\bm\pi\dot{\bm P}\bm 1)^2+2\bm\pi\dot{\bm P}(\bm Z-\bm1\bm\pi)\dot{\bm P}\bm 1,\\
\bar\mu:=\bar{\bm\pi}\dot{\bar{\bm P}}{\bm1},\qquad\bar{\sigma}^2&:=&\bar{\bm\pi}\ddot{\bar{\bm P}}\bm 1-(\bar{\bm\pi}\dot{\bar{\bm P}}\bm 1)^2+2\bar{\bm\pi}\dot{\bar{\bm P}}(\bar{\bm Z}-\bm1\bar{\bm\pi})\dot{\bar{\bm P}}\bm 1,
\end{eqnarray*}
where $\bm1$ is the column vector of 1s of the appropriate dimension and $\dot{\bm P}$ is obtained from $\bm P$, $\ddot{\bm P}$ from $\dot{\bm P}$, $\dot{\bar{\bm P}}$ from $\bar{\bm P}$, and $\ddot{\bar{\bm P}}$ from $\dot{\bar{\bm P}}$ by replacing each $q_m$ by $-q_m$ for $m=0,1,2,3,4$.  This ``rule of thumb'' requires some caution:  It must be applied before any simplifications to $\bm P$ or $\bar{\bm P}$ are made using $q_m=1-p_m$.  Of course, $\bm\pi$ and $\bar{\bm\pi}$ are the unique stationary distributions, and $\bm Z$ and $\bar{\bm Z}$ are the fundamental matrices, of $\bm P$ and $\bar{\bm P}$.

\begin{theorem}\label{means,variances}
\begin{equation}\label{means}
\mu^\circ=\mu=\bar\mu
\end{equation}
and
\begin{equation}\label{variances}
(\sigma^\circ)^2=\sigma^2=\bar\sigma^2.
\end{equation}
\end{theorem}

\begin{remark}
Before proving this, let us explain its significance.  $\mu^\circ$ and $(\sigma^\circ)^2$ are the mean and variance that appear in the SLLN and the CLT.  They are defined in terms of $\bm P^\circ$, the augmented one-step transition matrix.  $\mu$ and $\sigma^2$ are defined analogously in terms of $\bm P$, the original one-step transition matrix.  $\bar{\mu}$ and $\bar{\sigma}^2$ are defined analogously in terms of $\bar{\bm P}$, the reduced one-step transition matrix.  $\mu^\circ$ and $(\sigma^\circ)^2$ are easiest to interpret, whereas $\bar{\mu}$ and $\bar{\sigma}^2$ are easiest to evaluate.

To emphasize that $\bm P$ of \eqref{x to x^ij} and \eqref{x to x} describes game $B$, we denote it by $\bm P_B$, and we denote its unique stationary distribution by $\bm\pi_B$.  $\bm P_A$ is the special case in which $p_0=p_1=p_2=p_3=p_4=1/2$.  For game $C:=\gamma A+(1-\gamma)B$, we denote the one-step transition matrix by $\bm P_{(\gamma,1-\gamma)}:=\gamma \bm P_A+(1-\gamma)\bm P_B$.  This is just $\bm P$ but with $\bm p=(p_0,p_1,p_2,p_3,p_4)$ replaced by $\bm p'=(p_0',p_1',p_2',p_3',p_4')$, where $p_m':=\gamma(1/2)+(1-\gamma)p_m$ for $m=0,1,2,3,4$.  Thus, Theorems \ref{SLLN/CLT} and \ref{means,variances} apply to game $B$ (using $\bm p$) and to game $C:=\gamma A+(1-\gamma)B$ (using $\bm p'$).
\end{remark}

\begin{proof}
To emphasize the fact that $P^\circ((\bm x,s),(\bm y,t))$ does not depend on $s$, we write it temporarily as $P^\circ((\bm x,\cdot),(\bm y,t))$.  This leads to
\begin{equation}\label{mean-eq}
\mu^\circ=\bm\pi^\circ\dot{\bm P}^\circ\bm1=\sum_{\bm x,s,\bm y,t}\pi^\circ(\bm x,s)\dot{P}^\circ((\bm x,\cdot),(\bm y,t))=\sum_{\bm x,\bm y}\pi(\bm x)\dot{P}(\bm x,\bm y)=\bm\pi\dot{\bm P}\bm1=\mu.
\end{equation}
To show that $(\sigma^\circ)^2=\sigma^2$, we need to show that
\begin{equation*}
\bm\pi^\circ\dot{\bm P}^\circ(\bm Z^\circ-\bm1\bm\pi^\circ)\dot{\bm P}^\circ\bm 1=\bm\pi\dot{\bm P}(\bm Z-\bm1\bm\pi)\dot{\bm P}\bm 1.
\end{equation*}
Now, by Kemeny and Snell \cite[p.~75]{KS76},
$\bm Z-\bm1\bm\pi=\sum_{m=1}^\infty(\bm P^{m-1}-\bm1\bm\pi)$,
so it is enough to show that
$$
\bm\pi^\circ\dot{\bm P}^\circ((\bm P^\circ)^{m-1}-\bm1\bm\pi^\circ)\dot{\bm P}^\circ\bm 1=\bm\pi\dot{\bm P}(\bm P^{m-1}-\bm1\bm\pi)\dot{\bm P}\bm 1,\qquad m\ge1,
$$
or that
$$
\bm\pi^\circ\dot{\bm P}^\circ(\bm P^\circ)^{m-1}\dot{\bm P}^\circ\bm 1=\bm\pi\dot{\bm P}\bm P^{m-1}\dot{\bm P}\bm 1,\qquad m\ge1.
$$
Given $m\ge1$, we have
\begin{eqnarray*}
&&\!\!\!\!\!\!\bm\pi^\circ\dot{\bm P}^\circ(\bm P^\circ)^{m-1}\dot{\bm P}^\circ\bm 1\nonumber\\
&=&\sum_{\bm x,s,\bm y,t,\bm z,u,\bm w,v}\pi^\circ(\bm x,s)\dot{P}^\circ((\bm x,\cdot),(\bm y,t))(P^\circ)^{m-1}((\bm y,\cdot),(\bm z,u))\dot{P}^\circ((\bm z,\cdot),(\bm w,v))\nonumber\\
&=&\sum_{\bm x,\bm y,t,\bm z,\bm w,v}\pi(\bm x)\dot{P}^\circ((\bm x,\cdot),(\bm y,t))P^{m-1}(\bm y,\bm z)\dot{P}^\circ((\bm z,\cdot),(\bm w,v))\nonumber\\
&=&\sum_{\bm x,\bm y,\bm z,\bm w}\pi(\bm x)\dot{P}(\bm x,\bm y)P^{m-1}(\bm y,\bm z)\dot{P}(\bm z,\bm w)\nonumber\\
&=&\bm\pi\dot{\bm P}{\bm P}^{m-1}\dot{\bm P}\bm 1.
\end{eqnarray*}

The second equalites in \eqref{means} and \eqref{variances} are now clear because it is enough that
$$
\bm\pi\dot{\bm P}\bm 1=\bar{\bm\pi}\dot{\bar{\bm P}}\bm 1,\qquad
\bm\pi\dot{\bm P}\bm P^{m-1}\dot{\bm P}\bm 1=\bar{\bm\pi}\dot{\bar{\bm P}}\bar{\bm P}^{m-1}\dot{\bar{\bm P}}\bm 1,\quad m\ge1,
$$
and this is a consequence of Eq.~(9) in Ref.~\cite{EL12b}.
\end{proof}

Next we need versions of the SLLN and the CLT suited to game $C:=A^rB^s$.  The key result is Theorem 6 of Ref.~\cite{EL09}.  

For the same reason as before, the theorem does not apply directly to $\bm P_B$ (given by \eqref{x to x^ij} and \eqref{x to x}) and $\bm P_A$ (the special case of $\bm P_B$ with $p_0=p_1=p_2=p_3=p_4=1/2$).  Therefore we again consider the Markov chain in $\Sigma^\circ$ with one-step transition matrix $\bm P^\circ$ as in \eqref{x to x^ij,1}--\eqref{x to x,-1}.  We let $\bm P_B^\circ=\bm P^\circ$ and $\bm P_A^\circ$ be $\bm P^\circ$ with $p_0=p_1=p_2=p_3=p_4=1/2$.  With $\bm W^\circ$ as before, the theorem applies.

\begin{theorem}\label{SLLN/CLT2}
Fix $r,s\ge1$.  Assume that $\bm P^\circ:=(\bm P_A^\circ)^r(\bm P_B^\circ)^s$, as well as all cyclic permutations of $(\bm P_A^\circ)^r(\bm P_B^\circ)^s$, are ergodic, and let the row vector $\bm\pi^\circ$ be the unique stationary distribution of $\bm P^\circ$.  Let
\begin{equation*}
\mu_{[r,s]}^\circ:=\frac{1}{r+s}\sum_{v=0}^{s-1}\bm\pi^\circ(\bm P_A^\circ)^r(\bm P_B^\circ)^v\dot{\bm P}_B^\circ\bm1
\end{equation*}
and
\begin{eqnarray*}
&&\!\!\!(\sigma_{[r,s]}^\circ)^2\\
&&{}=1-\frac{1}{r+s}\sum_{v=0}^{s-1}(\bm\pi(\bm P_A^\circ)^r(\bm P_B^\circ)^v\dot{\bm P}_B^\circ\bm1)^2\\
&&\quad{}+\frac{2}{r+s}\bigg[\sum_{u=0}^{r-1}\sum_{v=0}^{s-1}\bm\pi^\circ(\bm P_A^\circ)^u\dot{\bm P}_A^\circ((\bm P_A^\circ)^{r-u-1}(\bm P_B^\circ)^v-\bm1\bm\pi^\circ(\bm P_A^\circ)^r(\bm P_B^\circ)^v)\dot{\bm P}_B^\circ\bm1\nonumber\\ 
&&\qquad{}+\sum_{0\le u<v\le s-1}\bm\pi^\circ(\bm P_A^\circ)^r(\bm P_B^\circ)^u\dot{\bm P}_B^\circ((\bm P_B^\circ)^{v-u-1}-\bm1\bm\pi^\circ(\bm P_A^\circ)^r(\bm P_B^\circ)^v)\dot{\bm P}_B^\circ\bm1\\
&&\qquad{}+\sum_{u=0}^{r-1}\sum_{v=0}^{s-1}\bm\pi^\circ(\bm P_A^\circ)^u\dot{\bm P}_A^\circ(\bm P_A^\circ)^{r-u-1}(\bm P_B^\circ)^s(\bm Z^\circ-\bm1\bm\pi^\circ)(\bm P_A^\circ)^r(\bm P_B^\circ)^v\dot{\bm P}_B^\circ\bm1\\
&&\qquad{}+\sum_{u=0}^{s-1}\sum_{v=0}^{s-1}\bm\pi^\circ(\bm P_A^\circ)^r(\bm P_B^\circ)^u\dot{\bm P}_B^\circ(\bm P_B^\circ)^{s-u-1}(\bm Z^\circ-\bm1\bm\pi^\circ)(\bm P_A^\circ)^r(\bm P_B^\circ)^v\dot{\bm P}_B^\circ\bm1\bigg].
\end{eqnarray*}
Let $\{X_n^\circ\}_{n\ge0}$ be a temporally nonhomogeneous Markov chain in $\Sigma^\circ$ with one-step transition matrices $\bm P_A^\circ,\ldots,\bm P_A^\circ$ $(r\text{ times})$, $\bm P_B^\circ,\ldots,\bm P_B^\circ$ $(s\text{ times})$, $\bm P_A^\circ,\ldots,\bm P_A^\circ$ $(r\text{ times})$, $\bm P_B^\circ,\ldots,\bm P_B^\circ$ $(s\text{ times})$, and so on, and let the initial distribution be arbitrary.  For each $n\ge1$, define $\xi_n:=w^\circ(X_{n-1}^\circ,X_n^\circ)$ and $S_n:=\xi_1+\cdots+\xi_n$.  Then $\lim_{n\to\infty}n^{-1}S_n=\mu_{[r,s]}^\circ$ \emph{a.s.} and, if $(\sigma_{[r,s]}^\circ)^2>0$, then 
$$
\frac{S_n-n\mu_{[r,s]}^\circ}{\sqrt{n(\sigma_{[r,s]}^\circ)^2}}\to_d N(0,1) \text{ as } n\to\infty.
$$
\end{theorem}

Again there are simpler expressions for this mean and variance.  We let $\dot{\bm P}_B$ be obtained from $\bm P_B$ and $\dot{\bar{\bm P}}_B$ from $\bar{\bm P}_B$ by replacing each $q_m$ by $-q_m$ for $m=0,1,2,3,4$.  $\dot{\bm P}_A$ and $\dot{\bar{\bm P}}_A$ are the special case $p_0=p_1=p_2=p_3=p_4=1/2$.  We define $\mu_{[r,s]}$ and $\bar\mu_{[r,s]}$ in terms of $\bm\pi$, $\bm P_A$, and $\bm P_B$, and in terms of $\bar{\bm\pi}$, $\bar{\bm P}_A$, and $\bar{\bm P}_B$ in the same way that $\mu_{[r,s]}^\circ$ was defined in terms of $\bm\pi^\circ$, $\bm P_A^\circ$, and $\bm P_B^\circ$.  Finally, $\sigma_{[r,s]}^2$ and $\bar\sigma_{[r,s]}^2$ are defined analogously to $(\sigma_{[r,s]}^\circ)^2$.

\begin{theorem}
\begin{equation}\label{means-pattern}
\mu_{[r,s]}^\circ=\mu_{[r,s]}=\bar{\mu}_{[r,s]}
\end{equation}
and 
\begin{equation}\label{variances-pattern}
(\sigma_{[r,s]}^\circ)^2=\sigma_{[r,s]}^2=\bar{\sigma}_{[r,s]}^2.
\end{equation}
\end{theorem} 

\begin{proof}
The first equation in \eqref{means-pattern} follows exactly as in \eqref{mean-eq}, while the second uses a result of Ref.~\cite{EL12b}.  The first equation in \eqref{variances-pattern} is proved in the same way as the first equation in \eqref{variances}, while the second uses the method of Ref.~\cite{EL12b}.
\end{proof}

We conclude this section with an application of the preceding SLLNs. Let us denote the means above by $\mu_B(\bm p)$, $\mu_{(\gamma,1-\gamma)}(\bm p)$, and $\mu_{[r,s]}(\bm p)$ to emphasize their dependence on the probability parameters.  The proof of the following is essentially the same as in Ref.~\cite{EL12a}.

\begin{corollary}\label{symm-corollary}
With $q_m:=1-p_m$ for $m=0,1,2,3,4$, we have
\begin{eqnarray*}
\mu_B(p_0,p_1,p_2,p_3,p_4)&=&-\mu_B(q_4,q_3,q_2,q_1,q_0),\\
\mu_{(\gamma,1-\gamma)}(p_0,p_1,p_2,p_3,p_4)&=&-\mu_{(\gamma,1-\gamma)}(q_4,q_3,q_2,q_1,q_0),\\
\mu_{[r,s]}(p_0,p_1,p_2,p_3,p_4)&=&-\mu_{[r,s]}(q_4,q_3,q_2,q_1,q_0).
\end{eqnarray*}
\end{corollary}

This result shows that the Parrondo region ($\mu_B\le0$ and $\mu_{(\gamma,1-\gamma)}>0$, or $\mu_B\le0$ and $\mu_{[r,s]}>0$) and the anti-Parrondo region ($\mu_B\ge0$ and $\mu_{(\gamma,1-\gamma)}<0$, or $\mu_B\ge0$ and $\mu_{[r,s]}<0$) are images of each other under the affine transformation $(p_0,p_1,p_2,p_3,p_4)\mapsto(q_4,q_3,q_2,q_1,q_0)$, hence have equal volumes.

\section{Reducible cases}\label{reducible}

We have assumed that $0<p_m<1$ for $m=0,1,2,3,4$, which ensures that our Markov chain is irreducible and aperiodic, but this assumption can be weakened.  Let us continue to assume that $0<p_m<1$ for $m=1,2,3$ but not necessarily for $m=0$ or $m=4$.   We denote by $\bm0\in\Sigma$ the state consisting of all 0s, and by $\bm1\in\Sigma$ the state consisting of all 1s.\medskip

\begin{enumerate}
\item Suppose $p_0=1$ and $0<p_4<1$.  Then state $\bm 0$ cannot be reached from $\Sigma-\{\bm 0\}$ and $\bm P$, with row $\bm 0$ and column $\bm 0$ deleted, is a stochastic matrix that is irreducible and aperiodic.

\item Suppose $p_0=0$ and $0\le p_4<1$.  Then state $\bm 0$ is absorbing, and absorption eventually occurs with probability 1.  Hence $S_n-S_{n-1}=-1$ for all $n$ sufficiently large, so $\mu_B=-1$.

\item Suppose $0<p_0<1$ and $p_4=0$.  This is analogous to case 1, with $\bm1$ in place of $\bm0$.

\item Suppose $0<p_0\le 1$ and $p_4=1$.  This is analogous to case 2, with $\bm1$ in place of $\bm0$ and 1 in place of $-1$.

\item Suppose $p_0=1$ and $p_4=0$.  Then states $\bm 0$ and $\bm 1$ cannot be reached from $\Sigma-\{\bm 0,\bm 1\}$ and $\bm P$, with rows $\bm 0$ and $\bm 1$ and columns $\bm 0$ and $\bm 1$ deleted, is a stochastic matrix.  If $M$ or $N$ is odd, it appears that irreducibility and aperiodicity hold, but we do not have a proof.  If $M$ and $N$ are even, then the two states having a checkerboard pattern of 0s and 1s are absorbing while all other states are transient.  From  either absorbing state there is a win of one unit with probability 1/2 and a loss of one unit with probability 1/2.  Consequently, $\mu_B=0$, regardless of $p_1$, $p_2$, and $p_3$.  For $p_1,p_2,p_3>1/2$, we expect that the parameter vector $(1,p_1,p_2,p_3,0)$ belongs to the Parrondo region (the $(1/2,1/2)$ random mixture version), but we do not have a proof.

\item Suppose $p_0=0$ and $p_4=1$.  Then both $\bm 0$ and $\bm 1$ are absorbing, and absorption occurs with probability 1.  The probability of absorption at $\bm 1$ depends on the initial state (or equivalence class), and can be calculated for small $M,N$.  Of course, $\mu_B$ is undefined in this case.
\end{enumerate}

A more complete analysis would also allow $p_1$, $p_2$, and $p_3$ to be 0 or 1, but conditions for irreducibility would be quite complicated, so we do not pursue it.  However, we do allow $p_1$, $p_2$, and $p_3$ to be 0 or 1 in the examples below.

\section{Monotonicity of the mean function}\label{mono-Section}

In game $B$, one would expect that, if one of the five coins were replaced by a coin with a higher probability of heads, the mean profit at equilibrium would increase.  Paradoxically, as we will see, this is false in general.  However, it is true under the extra assumption that the probability of heads is monotone nondecreasing in the number of winners, that is, $p_0\le p_1\le p_2\le p_3\le p_4$.

\begin{theorem}\label{monotonicity}
The mean profit function $\mu_B(\bm p)$ is monotone nondecreasing in each variable $p_0,p_1,p_2,p_3,p_4$ on the subset of the parameter space $[0,1]^5$ on which 
\begin{equation}\label{mono}
p_0\le p_1\le p_2\le p_3\le p_4,
\end{equation}
excluding only the case, $p_0=0$ and $p_4=1$, in which $\mu_B(\bm p)$ is undefined.
\end{theorem}

\begin{proof}
Given the parameter vector $\bm p=(p_0,p_1,p_2,p_3,p_4)\in[0,1]^5$ and the initial state $\bm x(0)$, we can simulate our Markov chain, based on three independent sequences of i.i.d.\ random variables, $I_1,I_2,\ldots$ uniform on $\{1,2,\ldots,M\}$, $J_1,J_2,\ldots$ uniform on $\{1,2,\ldots,N\}$, and $U_1,U_2,\ldots$ uniform on $(0,1)$.  The interpretation is that $(I_k,J_k)$ is the site of the player chosen to play at round $k$ and $U_k$ determines whether a win or a loss occurs.  More precisely, we define $\bm X(k)$ recursively in terms of $\bm X(k-1)$, $I_k$, $J_k$, and $U_k$.  Specifically, 
$$
X_{i,j}(k)=\begin{cases}X_{i,j}(k-1)&\text{if $(i,j)\ne (I_k,J_k)$},\cr
1\{U_k\le p_{m_{i,j}(\bm X(k-1))}\}&\text{if $(i,j)=(I_k,J_k)$}.
\end{cases}
$$
Then, letting 
$$
S_n=\sum_{k=1}^n [2\cdot 1\{U_k\le p_{m_{I_k,J_k}(\bm X(k-1))}\}-1],
$$
we see that $S_n$ is the players' cumulative profit after $n$ rounds, hence
$\mu_B(\bm p)=\lim_{n\to\infty} n^{-1}S_n$  a.s.  (We exclude only the case, $p_0=0$ and $p_4=1$, in which $\mu_B(\bm p)$ is undefined.)

Next, let us couple two such processes with the same starting point but different parameter vectors, $\bm p=(p_0,p_1,p_2,p_3,p_4)$ and $\bm p'=(p_0',p_1',p_2',p_3',p_4')$, using the same $I_1,I_2,\ldots$, $J_1,J_2,\ldots$, and $U_1,U_2,\ldots$ sequences.  However, let us assume that the first of the two parameter vectors is monotone, that is, \eqref{mono} holds.  We also assume that
$$
p_0\le p_0',\quad p_1\le p_1',\quad p_2\le p_2',\quad p_3\le p_3',\quad p_4\le p_4'.
$$
We now define the $\bm X$ process and the $\bm X'$ process recursively as before, and we find that $S_n\le S_n'$ for all $n\ge0$, hence
$$
\mu_B(\bm p)=\lim_{n\to\infty}n^{-1}S_n\le\lim_{n\to\infty}n^{-1}S_n'=\mu_B(\bm p').
$$
This completes the proof.
\end{proof}

For an example in which the mean profit function $\mu_B(\bm p)$ fails to be nondecreasing, consider the case $M=N=3$ and
$$
\mu_B(1,0,9/10,1/2,1/2)\approx0.0554176,\quad \mu_B(1,0,1,1/2,1/2)=-1/3.
$$
More generally, the mean function $\mu_B(1,0,p_2,1/2,1/2)$ is not nondecreasing; indeed, its graph is displayed in Figure \ref{graph} with that of $\mu_{(1/2,1/2)}(1,0,p_2,1/2,1/2)$.

\begin{figure}[htb]
\centering
\includegraphics[width = 3in]{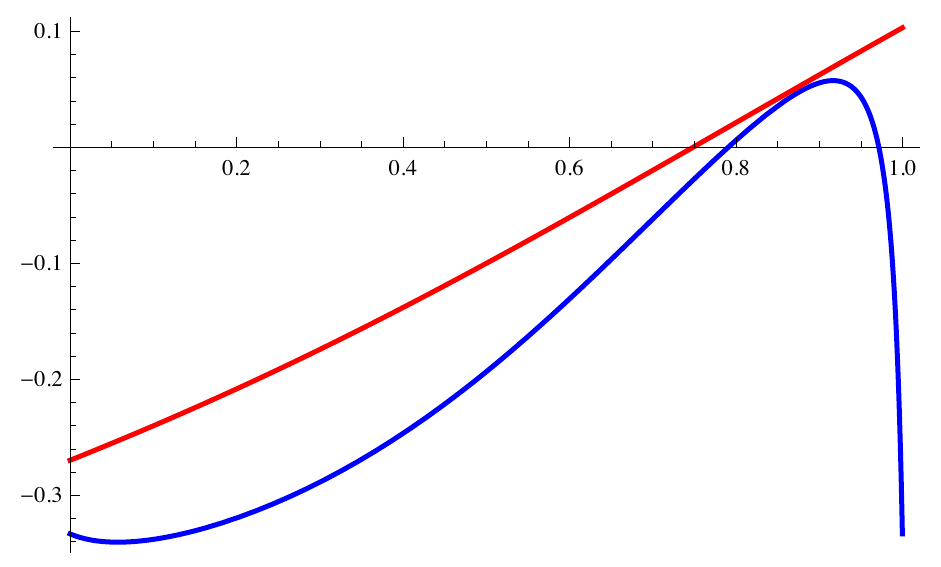}
\caption{\label{graph}The graphs of the mean profit functions $\mu_B(1,0,p_2,1/2,1/2)$ (blue) and $\mu_{(1/2,1/2)}(1,0,p_2,1/2,1/2)$ (red) as functions of $p_2\in[0,1]$.  Notice that the function $\mu_B(1,0,p_2,1/2,1/2)$ is decreasing on $[0.917,1]$.}
\end{figure}

To better understand why $\mu_B(1,0,p_2,1/2,1/2)$ is decreasing in $p_2$ for $0.917\le p_2\le 1$, it may help to rewrite $\mu_B(\bm p)$ as
\begin{eqnarray*}
\mu_B(\bm p)&=&\bm\pi\dot{\bm P}_B\bm1=(MN)^{-1}\sum_{i=1}^M\sum_{j=1}^N\sum_{\bm x}\pi(\bm x)(2p_{m_{i,j}(\bm x)}-1)\\
&=&\sum_{\bm x}\pi(\bm x)(2p_{m_{2,2}(\bm x)}-1)=\sum_{m=0}^4\lambda_m(\bm p)(2p_m-1),
\end{eqnarray*}
where $\lambda_m(\bm p):=\sum_{\bm x: m_{2,2}(\bm x)=m}\pi(\bm x)$ is the distribution of $m_{2,2}$ at equilibrium.  Thus, $\mu_B(\bm p)$ is a weighted average of the means $2p_m-1$ ($m=0,1,2,3,4$), but of course the weights $\lambda_m(\bm p)$ are functions of $\bm p$.  In our example, we can graph the weights as functions of $p_2\in[0,1]$ with the results shown in Figure \ref{weights}.

\begin{figure}[htb]
\centering
\includegraphics[width = 3.3in]{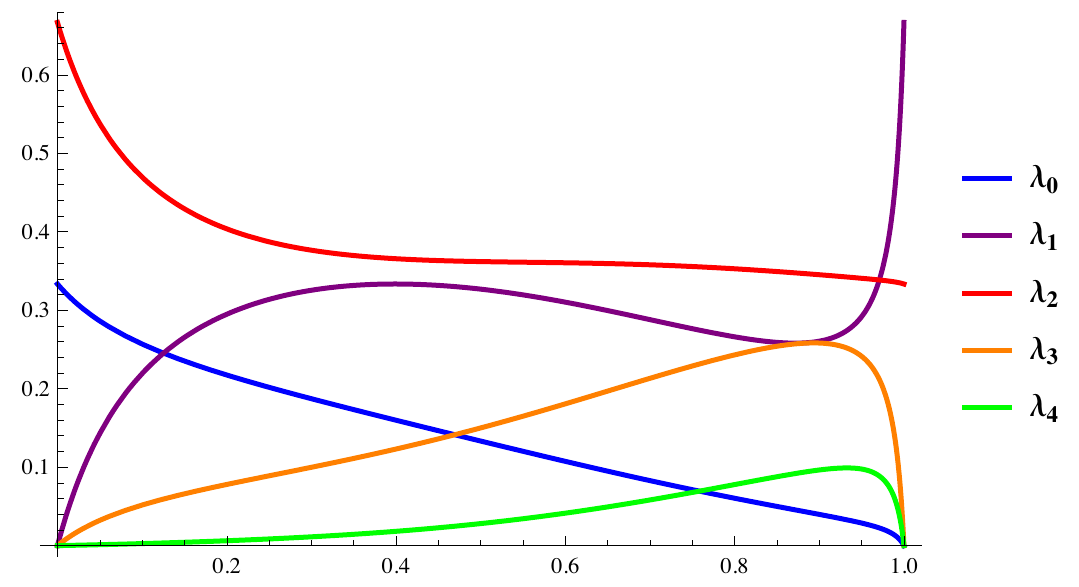}
\caption{\label{weights}The graphs of $\lambda_m(1,0,p_2,1/2,1/2)$ for $m=0,1,2,3,4$ as functions of $p_2\in[0,1]$.  To distinguish the curves without relying on color, at $\bm p=(1,0,3/5,1/2,1/2)$ we have $\lambda_4(\bm p)<\lambda_0(\bm p)<\lambda_3(\bm p)<\lambda_1(\bm p)<\lambda_2(\bm p)$.}
\end{figure}

The dramatic changes in the weights $\lambda_m(1,0,p_2,1/2,1/2)$ as $p_2$ increases from 9/10 to 1 is a consequence of the fact that the Markov chain is irreducible and aperiodic for $0<p_2<1$ but absorbing states appear at $p_2=1$.  Indeed, all six states with one row or column containing all 1s and the remaining six entries being 0s are absorbing because $p_1=0$ and $p_2=1$.  The weights become $(0,2/3,1/3,0,0)$, hence $\mu_B(1,0,1,1/2,1/2)=(2/3)(-1)+(1/3)(1)=-1/3$.  Incidentally, a similar phenomenon occurs at $p_2=0$, where the weights are $(1/3,0,2/3,0,0)$ and $\mu_B(1,0,0,1/2,1/2)=(1/3)(1)+(2/3)(-1)=-1/3$.

Examples of the nonmonotonicity of $\mu_B(\bm p)$ appear not just when $M=N=3$.  In the case $M=N=4$, we have a similar phenomenon when $(p_0,p_1,p_3,p_4)=(3/4,1,1,1/4)$, namely that $\mu_B(3/4,1,p_2,1,1/4)$ is decreasing in $p_2$ for $p_2$ large enough.  Here the effect is less dramatic, perhaps because there are no absorbing states when $p_2=1$.

On the other hand, in the one-dimensional spatial model with $N\ge3$ players and probability parameters $p_0$, $p_1$, and $p_2$, extensive computations suggest that monotonicity always holds.  In the simplest case, $N=3$, this can be proved algebraically.  With $p_0\le p_1\le p_2$ in place of \eqref{mono}, the analogue of Theorem \ref{monotonicity} holds.  But a general proof remains elusive.

\section{Convergence of mean profit}\label{convergence}

In the one-dimensional case we found sufficient conditions for the mean profits $\mu_B$ and $\mu_{(1/2,1/2)}$ to converge as $N\to\infty$.  In that case there were only three parameters, $p_0,p_1,p_2\in[0,1]$ (assuming that the coin tossed depends only on the number of winners among the nearest neighbors), and we found that $\mu_{(1/2,1/2)}$ always converges, whereas $\mu_B$ converges on a subset of the parameter space with volume $3323/4032\approx0.824$.  We used four methods to get this result, the basic estimate, attractiveness and repulsiveness, coalescing duality, and annihilating duality.  For example, the latter leads to the sufficient condition
\begin{equation*}
p_0,p_1,p_2\in(2\overline{p}-1,2\overline{p})\cap(0,1),\quad \overline{p}:=(p_0+2p_1+p_2)/4,
\end{equation*}
and the set of such $(p_0,p_1,p_2)\in(0,1)^3$ is a region of volume 2/3.  In general, the limit can be described in terms of an ergodic spin system.  

Analogous methods apply in the two-dimensional setting with $M,N\ge3$, but unfortunately the results obtained are less satisfactory.  For this reason we do not provide complete details but simply state the sufficient conditions for convergence of $\mu_B$.  Let a parameter vector $\bm p=(p_0,p_1,p_2,p_3,p_4)\in[0,1]^5$ be given.  Of the four methods mentioned above, only the first and last lead to useful results.  The basic estimate leads to the sufficient condition
\begin{equation}\label{basic}
\max_{0\le m\le3}|p_{m+1}-p_m|<\frac{1}{4},
\end{equation}
and annihilating duality leads to the sufficient condition
\begin{eqnarray}\label{annihilating-2d}
&&|p_0 + 4 p_1 + 6 p_2 + 4 p_3 + p_4 - 8| + 4 |p_0 + 2 p_1 - 2 p_3 - p_4|+ 6 |p_0 - 2 p_2 + p_4|\nonumber\\
&&\quad{}  + 4 |p_0 - 2 p_1 + 2 p_3 - p_4|+ |p_0 - 4 p_1 + 6 p_2 - 4 p_3 + p_4|<8.
\end{eqnarray}
In both cases $\mu_B$ converges either as $M,N\to\infty$ or, for fixed $M\ge3$, as $N\to\infty$.  It follows that $\mu_{(1/2,1/2)}$ converges at $\bm p=(p_0,p_1,p_2,p_3,p_4)$ if $\mu_B$ converges at $\bm p'=(p_0',p_1',p_2',p_3',p_4')$, where $p_m':=(1/2+p_m)/2$ for $m=0,1,2,3,4$.  

\begin{theorem}
The mean profit $\mu_B$ converges as $M,N\to\infty$ provided the spin system in $\{0,1\}^{{\bf Z}^2}$ with flip rates 
$$
c_{i,j}(\bm x):=\begin{cases}p_{m_{i,j}(\bm x)}&\text{if $x_{i,j}=0$,}\\q_{m_{i,j}(\bm x)}&\text{if $x_{i,j}=1$,}\end{cases}
$$
is ergodic.  A sufficient condition for ergodicity is that \eqref{basic} or \eqref{annihilating-2d} holds.
\end{theorem}

Condition \eqref{basic} is satisfied on a subset of the parameter space with volume 213/5120, or about 4.2\% of the time.  By simulation, condition \eqref{annihilating-2d} is satisfied about 8.8\% of the time.  At least one of the two conditions holds about 10.6\% of the time.

The corresponding conditions for game $C:=(1/2)(A+B)$ are as follows: the basic inequality holds on a set with volume 169/480 (about 35.2\% of the time);  annihilating duality applies about 74.8\% of the time; and at least one of the two conditions applies about 76.0\% of the time.

\section{The Parrondo region}\label{parrondo}

The Parrondo and anti-Parrondo regions for the one-dimensional model were described in Refs.~\cite{EL12a,EL12b}.  Analytical formulas for these regions were found for $N\le6$ and graphical representations can be obtained for $N\le9$ (cf.\ Ref.~\cite{EL15}).  Furthermore, the means $\mu_B$ and $\mu_C$ could be computed for arbitrary parameter vectors $(p_0,p_1,p_2)$ for $N\le19$ in the random mixture case and for $N\le18$ in the nonrandom pattern case (with $r+s\le4$).  As just noted, it was found analytically that $\mu_B$ converges on at least 82.4\% of the parameter space $(0,1)^3$, whereas $\mu_{(1/2,1/2)}$ converges on the entire parameter space $(0,1)^3$.  Moreover, we found empirically in the random mixture case that $\mu_B$ had stabilized to 3 or more significant digits by $N=19$, whereas $\mu_{(1/2,1/2)}$ had stabilized to 6 significant digits even earlier.  The point is that simulations for larger $N$ were unnecessary and would yield nothing useful.  For that reason no simulations of $\mu_B$ or $\mu_{(1/2,1/2)}$ were carried out in Refs.~\cite{EL12a,EL12b} for the one-dimensional model.
 
Results for the two-dimensional model are less satisfactory.  Let us assume that $M,N\ge3$.  Graphical representations of three-dimensional cross-sections of the Parrondo and anti-Parrondo regions are available only in the case $M=N=3$, and analytical formulas are not available.  The means $\mu_B$ and $\mu_{(1/2,1/2)}$ can be computed for arbitrary parameter vectors $\bm p=(p_0,p_1,p_2,p_3,p_4)$ provided $MN\le20$, but there are only six such cases with $3\le M\le N$ [$(M,N)=(3,3),(3,4),(3,5),(3,6),(4,4),(4,5)$].  Analytical conditions for convergence of $\mu_B$ and $\mu_{(1/2,1/2)}$ as $M,N\to\infty$ are much more restrictive (see Section~\ref{convergence}) than in one dimension, and numerical results have not stabilized by the time that exact computations are no longer feasible.  Thus, simulations become not only useful but necessary to fully understand the behavior of the model.

\begin{figure}[htb]
\centering
\includegraphics[width = 4.75in]{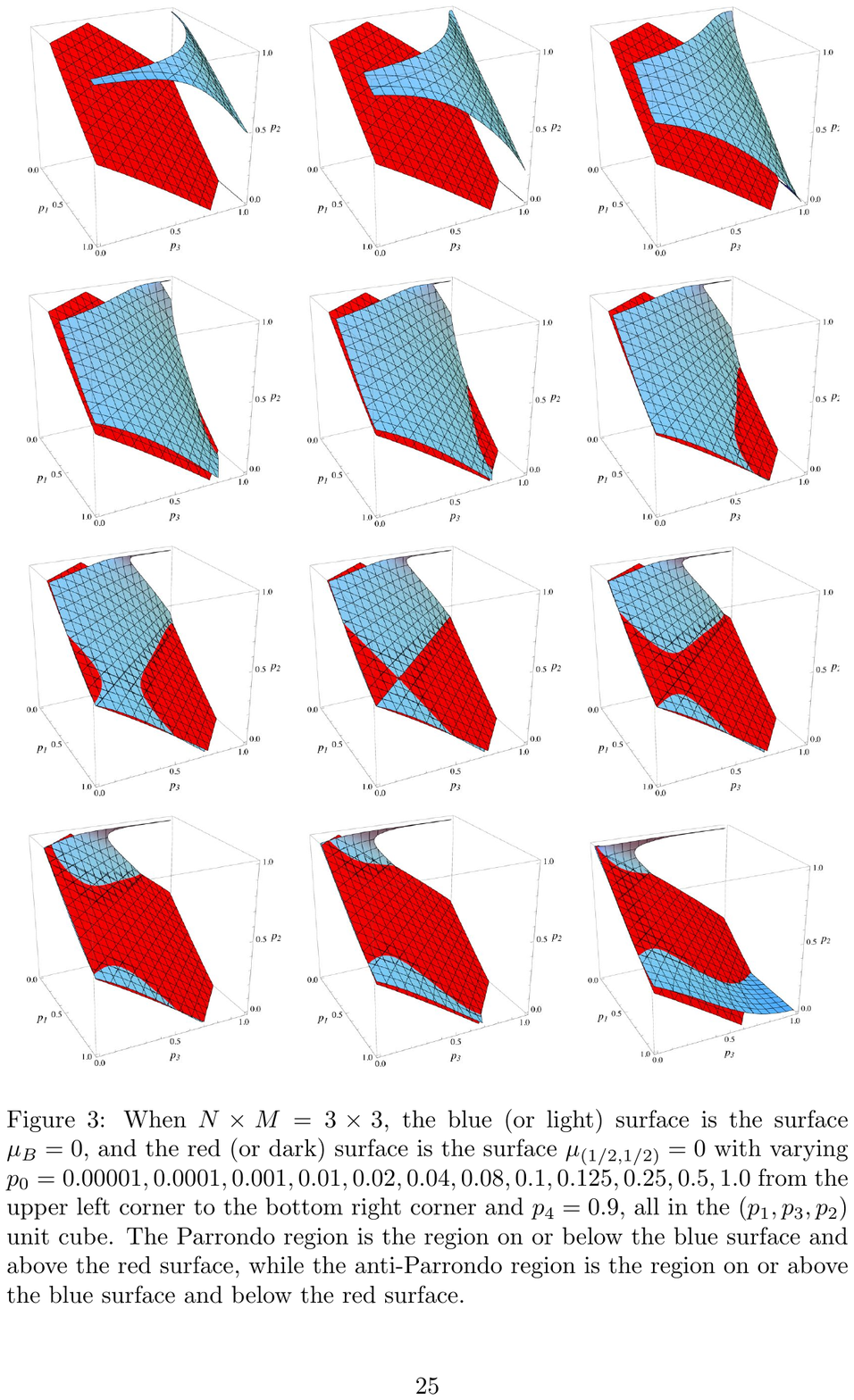}\quad

\caption{\label{p_4=0.9}When $(M,N)=(3,3)$, the blue (or light) surface is the surface $\mu_B=0$, and the red (or dark) surface is the surface $\mu_{(1/2,1/2)}=0$ with varying $p_0=0.00001, 0.0001, 0.001, 0.01, 0.02, 0.04, 0.08, 0.1, 0.125, 0.25, 0.5, 1.0$\ from the upper left corner to the bottom right corner and $p_4=0.9$, in the $(p_1,p_3,p_2)$ unit cube.  The three-dimensional cross-section of the Parrondo region is the region on or below the blue surface and above the red surface, while that of the anti-Parrondo region is the region on or above the blue surface and below the red surface.}
\end{figure}
\afterpage{\clearpage}

We begin with the case $M=N=3$, for which we can graph three-dimensional cross-sections of the Parrondo and anti-Parrondo regions for specified $p_0$ and $p_4$.  The one-step transition matrix $\bar{\bm P}$ is $26\times26$, allowing easy computation of the unique stationary distribution $\bar{\bm\pi}$ and hence the means $\mu_B$, $\mu_{(1/2,1/2)}$, and $\mu_{[2,2]}$.  The graphs of $\mu_B=0$ and $\mu_{(1/2,1/2)}=0$ are virtually indistinguishable from those of $\mu_B=0$ and $\mu_{[2,2]}=0$, and consequently only the former are displayed.  See Figure \ref{p_4=0.9}.  Denoting $\bar{\bm P}_B$ by \texttt{Pbar} and $\dot{\bar{\bm P}}_B\bm1$ by \texttt{Pbardotone}, the \textit{Mathematica} code needed to generate one of the cases in Figure \ref{p_4=0.9} has the following form, in which ellipses indicate omitted explicit formulas:

\begin{small}
\begin{verbatim}
p0 = 0.1; p4 = 0.9;
Pbar[p0_, p1_, p2_, p3_, p4_] := ...;
Pbardotone[p0_, p1_, p2_, p3_, p4_] := ...;
pibar = Array[x, {26}]; one = ConstantArray[1, {26}]; 
muB[p0_, p1_, p2_, p3_, p4_] := (solB = NSolve[{pibar == 
   pibar.Pbar[p0, p1, p2, p3, p4], pibar.one == 1}, pibar];
   mu1 = pibar.Pbardotone[p0, p1, p2, p3, p4] /. solB; 
   Return[mu1[[1, 1]]]);
muC[p0_, p1_, p2_, p3_, p4_] := muB[(p0 + 1/2)/2, 
   (p1 + 1/2)/2, (p2 + 1/2)/2, (p3 + 1/2)/2, (p4 + 1/2)/2];
Print[ContourPlot3D[{muB[p0, p1, p2, p3, p4] == 0, 
   muC[p0, p1, p2, p3, p4] == 0}, {p1, 0, 1}, {p3, 0, 1}, 
   {p2, 0, 1}, ContourStyle -> {RGBColor[135/255, 206/255, 235/255], 
   Red}, ViewPoint -> {3.3, -1.6, 1.7}]];
\end{verbatim}
\end{small}

We have chosen $p_4=0.9$ and a range of values of $p_0$, so that one can visually interpolate between successive figures to allow arbitrary $p_0$.  The Parrondo region is the region on or below the blue surface $\mu_B=0$ and above the red surface $\mu_{(1/2,1/2)}=0$.  Technically, the phrase ``below the blue surface'' is not quite accurate because $\mu_B(\bm p)$ is not necessarily increasing in $p_2$ (see Section \ref{mono-Section}).  What we mean is that the blue surface $\mu_B=0$ divides the unit cube into two regions, $\mu_B<0$ and $\mu_B>0$, and, for the most part, the region $\mu_B<0$ is below the blue surface $\mu_B=0$.  

Corresponding to each of the 12 cases in Figure \ref{p_4=0.9}, we have used simulation to estimate the volumes of the cross-sections of the Parrondo and anti-Parrondo regions.  See Table \ref{volumes}.

\begin{table}[ht]
\caption{\label{volumes}Simulated volumes of three-dimensional cross-sections of the Parrondo and anti-Parrondo regions corresponding to Figure \ref{p_4=0.9}.  Simulated estimates $\hat p$ are based on one million uniformly distributed points in the unit cube, with standard errors equal to $\sqrt{\hat p(1-\hat p)}/1000$.\medskip}
\catcode`@=\active \def@{\hphantom{0}}
\catcode`#=\active \def#{\hphantom{$-$}}
\begin{center}
\begin{footnotesize}
\begin{tabular}{lccccc}
\hline
\noalign{\smallskip}
& & \multicolumn{2}{c}{$C:=(1/2)(A+B)$} & \multicolumn{2}{c}{$C:=A^2B^2$}\\
\noalign{\smallskip}
\hline
\noalign{\smallskip}
$p_0$ & $p_4$ & Parrondo & anti-Parrondo & Parrondo & anti-Parrondo\\
      &       & region   &  region & region   &  region \\
\noalign{\smallskip}
\hline
\noalign{\smallskip}
0.00001 & 0.9  & 0.443631   & 0.000000 & 0.446307   &  0.000000 \\
0.0001 & 0.9   & 0.387100   & 0.000000 & 0.388498   &  0.000000 \\
0.001 & 0.9    & 0.274950   & 0.000000 & 0.276391   &  0.000000 \\
0.01 & 0.9     & 0.120368   & 0.000184 & 0.121233   &  0.000237 \\
0.02 & 0.9     & 0.077497   & 0.001474 & 0.077989   &  0.001648 \\
0.04 & 0.9     & 0.043804   & 0.005370 & 0.044782   &  0.005871 \\
0.08 & 0.9     & 0.019823   & 0.011948 & 0.020810   &  0.012661 \\
0.1  &  0.9    & 0.014627   & 0.014513 & 0.015121   &  0.015372 \\
0.125 & 0.9    & 0.010604   & 0.017140 & 0.010955   &  0.018200 \\
0.25  & 0.9    & 0.002872   & 0.025451 & 0.003123   &  0.026658 \\
0.5   & 0.9    & 0.002241   & 0.028884 & 0.002538   &  0.030795 \\
1.    & 0.9    & 0.011079   & 0.024876 & 0.012067   &  0.026160 \\
\noalign{\smallskip}
\hline
\end{tabular}
\end{footnotesize}
\end{center}
\end{table}

We were unsuccessful in trying to generate similar plots in the case $(M,N)=(3,4)$.  However, if we generate values of $\mu_B$ and $\mu_{(1/2,1/2)}$ at a grid of points $(p_1,p_3,p_2)$ with $p_0$ and $p_4$ fixed, and if we replace the instruction \texttt{ContourPlot3D} above by \texttt{ListContourPlot3D}, we get a figure that is a little less precise but still usable.  See Figure \ref{3x4ListContourPlot3D}.  Notice, for example, that the Parrondo region appears to have two connected components when $(M,N)=(3,3)$ and three connected components when $(M,N)=(3,4)$.

\begin{figure}[htb]
\centering
\includegraphics[width = 1.45in]{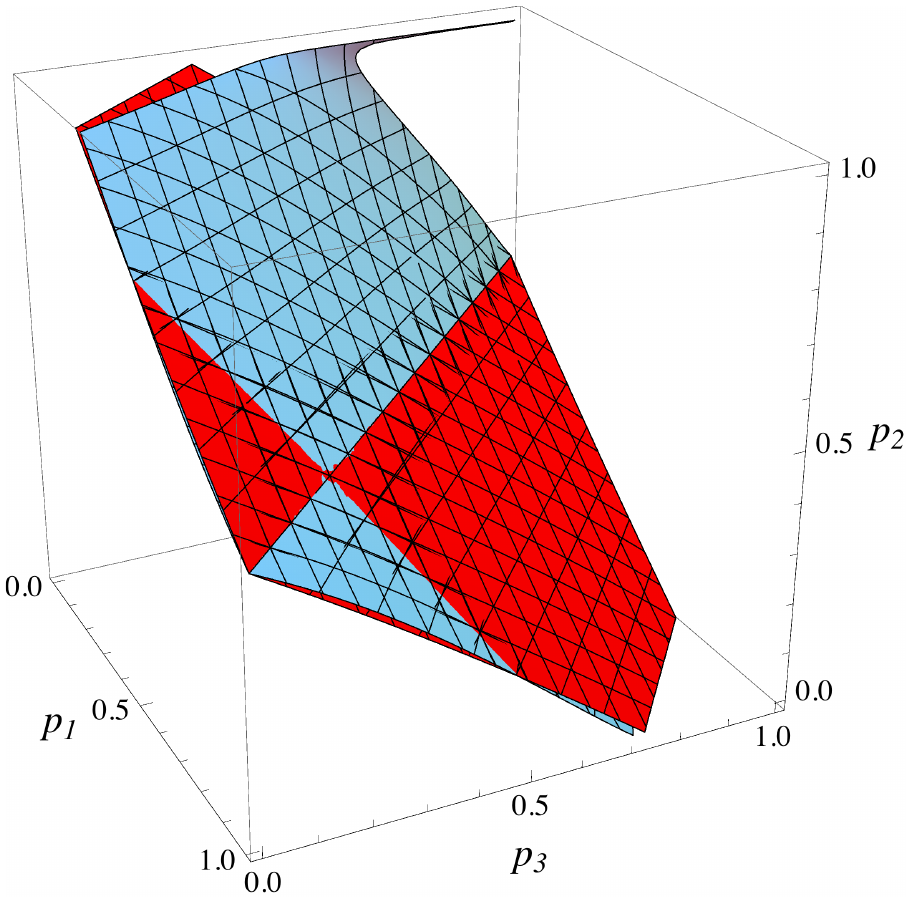}\quad
\includegraphics[width = 1.45in]{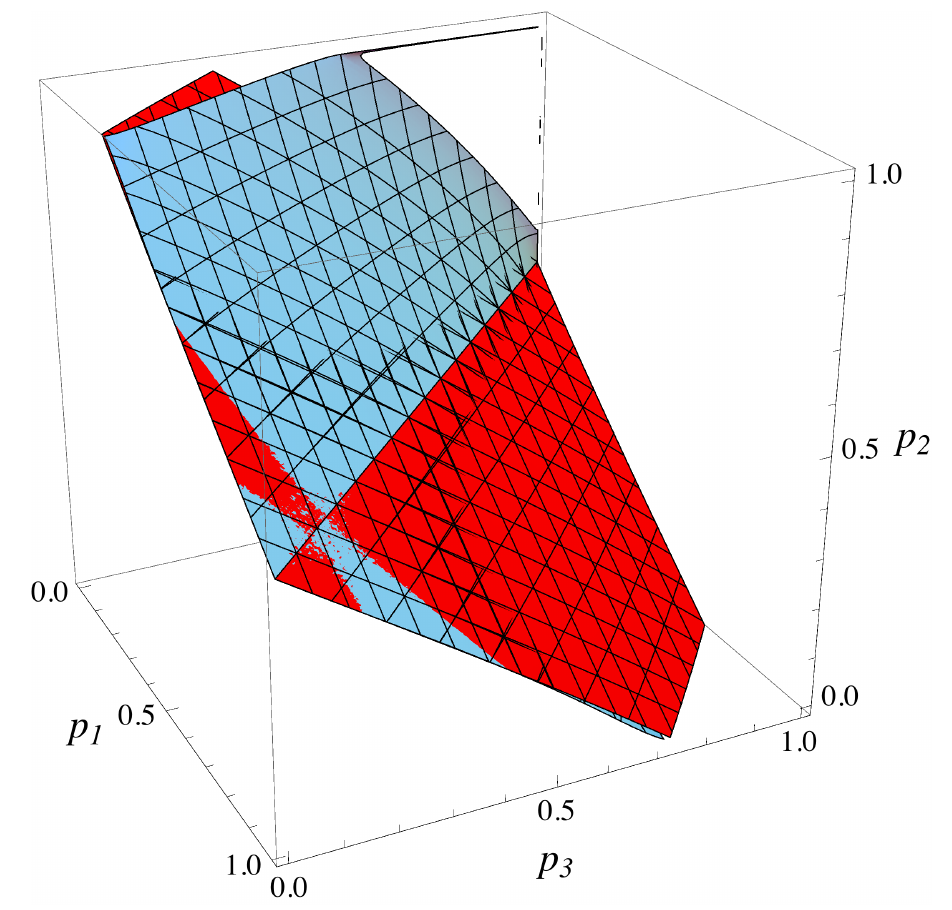}
\caption{\label{3x4ListContourPlot3D}When $(M,N)=(3,3)$ (left) and $(M,N)=(3,4)$ (right), the blue (or light) surface is the surface $\mu_B=0$, and the red (or dark) surface is the surface $\mu_{(1/2,1/2)}=0$ with $p_0=0.1$ and $p_4=0.9$, all in the $(p_1,p_3,p_2)$ unit cube.  The three-dimensional cross-section of the Parrondo region is the region on or below the blue surface and above the red surface, while that of the anti-Parrondo region is the region on or above the blue surface and below the red surface.}
\end{figure}

The method just described does not appear to extend to the case $(M,N)=(3,5)$ or $(M,N)=(4,4)$.  Instead, here we can generate two-dimensional cross-sections of the Parrondo and anti-Parrondo regions by evaluating the means at a grid of points with $p_0$, $p_1$, and $p_4$ fixed and $(p_3,p_2)=(i/100,j/100)$ with $i,j=0,1,\ldots,100$.  In the code above, we specify $p_1$ and replace the instruction \texttt{ContourPlot3D} above by \texttt{ListContourPlot}, eliminating \texttt{\{p1, 0, 1\}} and \texttt{ViewPoint -> \{3.3, -1.6, 1.7\}}.  See Figure \ref{4x4,p1}.

\begin{figure}[htb]
\centering
\includegraphics[width = 4.75in]{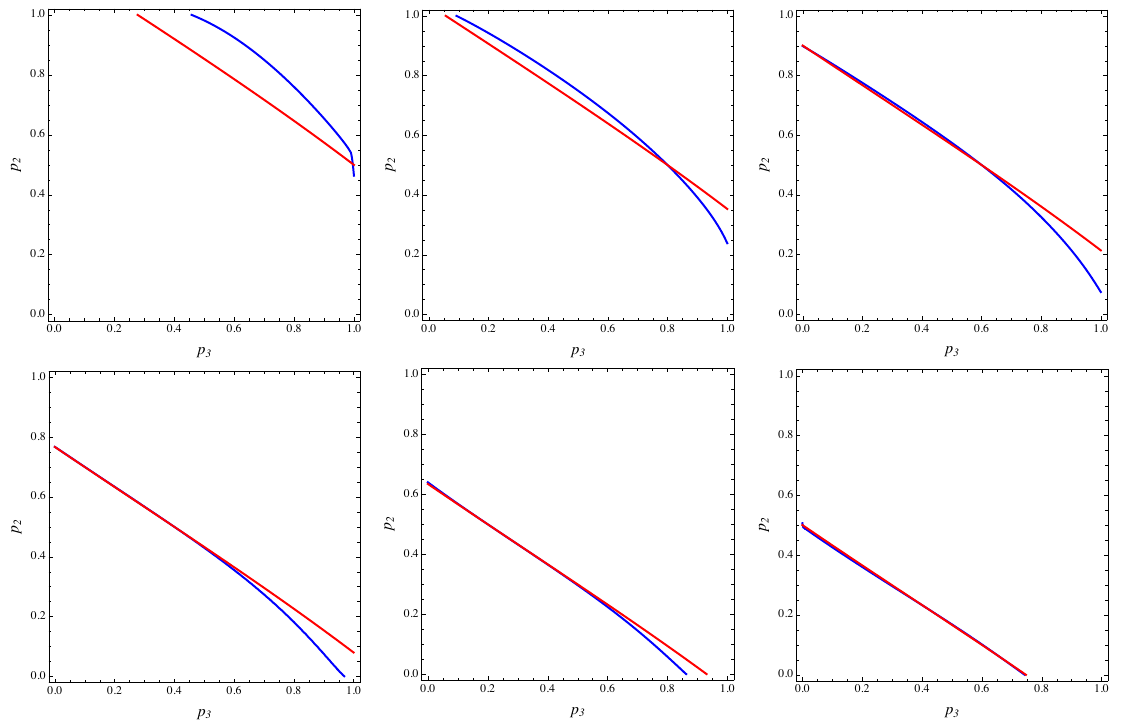}
\caption{\label{4x4,p1}When $(M,N)=(4,4)$, the blue curve is the curve $\mu_B=0$, and the red curve is the curve $\mu_{(1/2,1/2)}=0$ with $p_0=0.1$, varying $p_1=0.001$, 0.2, 0.4, 0.6, 0.8, 0.999 from upper left to lower right, and $p_4=0.9$, in the $(p_3,p_2)$ unit square.  The two-dimensional cross-section of the Parrondo region is the region on or below the blue curve and above the red curve, while that of the anti-Parrondo region is the region on or above the blue curve and below the red curve.}
\end{figure}

As we noted, there are only six pairs $(M,N)$ with $3\le M\le N$ in which exact computations are feasible.  The key quantity is the number of equivalence classes, and with $(M,N)=(4,5)$, there are 14676, which is similar to the number of equivalence classes in the one-dimensional model when $N=19$, namely 14310, and this is as far as we could compute the means in one dimension.  (The main issue is time, with memory being a minor issue.)

It is clear that, to proceed further, we must rely on simulation.  However, this presents us with its own set of complications.  By Theorem \ref{SLLN/CLT}, to estimate $\mu_B$, for example, it is enough to simulate the Markov chain with one-step transition matrix $\bm P^\circ$.  

How does one evaluate the standard error of such a simulated estimate?  By the CLTs in Theorems \ref{SLLN/CLT} and \ref{SLLN/CLT2}, the quantity of interest is $\sigma/\sqrt{n}$, but the variance parameters $\sigma^2$ of the CLT can be evaluated only for small $M$ and $N$.  Thus, we must simulate $\sigma^2$ as well.  In this regard, a relevant reference is Politis and Romano \cite{PR93}, who found a consistent estimator of $\sigma^2$ based on the first $n$ terms of a stationary strong-mixing sequence.

Let us now explain our simulation method.  Let $\xi_1, \xi_2, \ldots$ be a stationary strong-mixing sequence with mean and variance parameters
$$
\mu=\E[\xi_1]\quad\text{and}\quad\sigma^2=\Var(\xi_1)+2\sum_{i=1}^\infty\Cov(\xi_1,\xi_{i+1}).
$$
For our purposes, it will suffice to assume that $\{\xi_i\}$ is uniformly bounded and the coefficient of strong mixing decays geometrically fast.  Based on the first $n$ observations, how do we estimate $\sigma^2$?  Here is the approach of Politis and Romano.  Let $b$ be the block size and let $s_i:=\xi_i+\xi_{i+1}+\cdots+\xi_{i+b-1}$ be the sum of the terms in the $i$th block.  Then
\begin{eqnarray*}
\hat\sigma^2_{b,n}&:=&\frac{1}{n-b+1}\!\sum_{i=1}^{n-b+1}\bigg(\frac{1}{\sqrt{b}}\sum_{j=i}^{i+b-1}\xi_j-\sqrt{b}\,\bar\xi_n\bigg)^2=\frac{b}{n-b+1}\!\sum_{i=1}^{n-b+1}\bigg(\frac{s_i}{b}-\bar\xi_n\bigg)^2,
\end{eqnarray*}
where $\bar\xi_n:=(\xi_1+\xi_2+\cdots+\xi_n)/n$
and the sequence $s_1,s_2,\ldots,s_{n-b+1}$ can be evaluated recursively by keeping track of only the last $b$ observations.  Finally, letting $b:=\lfloor cn^{1/3}\rfloor$ for some constant $c$, this gives a consistent estimator of $\sigma^2$.
 
Another issue concerns the rate of convergence to equilibrium of the Markov chain.
This quantity is measured by the \textit{mixing time}, defined by
$$
t_{\text{mix}}(\eps)=\min\{n\ge1: {\sup}_x \|P^n(\bm x,\cdot)-\pi\|_{\text{TV}}\le\eps\},
$$
where $\|\mu-\nu\|_{\text{TV}}:=\frac{1}{2}\sum_{\bm x}|\mu(\bm x)-\nu(\bm x)|$ is the total variation norm.  For the Markov chain corresponding to game $A$ (or game $B$ and $p_0=p_1=p_2=p_3=p_4=1/2$) this has been evaluated.  That Markov chain is the so-called lazy random walk on the $MN$-dimensional ``hypercube'' $\{0,1\}^{MN}$, for which the mixing time satisfies (see Levin, Peres, and Wilmer \cite[p.~68]{LPW09})
$$
t_{\text{mix}}(\eps)\le MN[\log M+\log N+\log(1/\eps)].
$$
If we take $\eps=1/1000$, then $t_{\text{mix}}(0.001)\le MN(\log M+\log N+3\log10)$.  At a minimum we need $MN(\log M+\log N)$ steps for game $A$, presumably more for game $B$.

Recall that $\mu=\bm\pi\dot{\bm P}\bm1$, so letting $\mu_{\bm x,n}$ be the mean profit after $n$ steps of the Markov chain started in state $\bm x\in\Sigma$, then $\mu_{\bm x,n}=P^n(\bm x,\cdot)\dot{\bm P}\bm1$ (regarding $P^n(\bm x,\cdot)$ as a row vector) and
$$
|\mu_{\bm x,n}-\mu|\le |(P^n(\bm x,\cdot)-\bm\pi)\dot{\bm P}\bm1|\le\|P^n(\bm x,\cdot)-\bm\pi\|_1\|\dot{\bm P}\bm1\|_\infty,
$$
showing that $|\mu_{\bm x,n}-\mu|\le2\varepsilon$ if $n\ge t_{\text{mix}}(\eps)$.

The variance $\sigma^2$ can be computed for small $M$ and $N$, given the parameter vector $\bm p$, though it is too complicated to be computed algebraically.  However, there are some revealing special cases that are computable.  

With $\bm p=(0,1/4,1/2,3/4,1)$, our model is the so-called \textit{voter model}, which derives its name from the following interpretation: Each player is a voter, who has a position, 0 or 1, on a particular issue.  At each turn, a voter is chosen at random, and he then adopts the position of a randomly chosen nearest neighbor.  Mihailovi\'c and Rajkovi\'c \cite{MR06} studied the presence of the Parrondo effect for parameter vectors near that of the voter model.  (At $\bm p=(0,1/4,1/2,3/4,1)$ there are two absorbing states and $\mu_B$ does not exist; see Case 6 of Section \ref{reducible}.)  Thus, we consider an approximate voter model,
$\bm p=(\eps,(1/2+\eps)/2,1/2,(3/2-\eps)/2,1-\eps)$.
By Corollary \ref{symm-corollary}, $\mu_B=0$.  When $(M,N)=(3,3)$, we can compute the variance to be
$$
\sigma_B^2=\frac{9 - 17\eps + 12\eps^2 + 4\eps^3}{\eps(1 + 16\eps - 4\eps^2)}.
$$
This tends to $\infty$ as $\eps\to0$, showing that $\sigma^2(\bm p)$ may be unbounded in $\bm p$.

We illustrate our method by simulating the mean and variance parameters for various array sizes and a particular parameter vector, $\bm p=(1/20,3/20,8/13,\break 3/4,9/10)$.  We allow a warming-up period of order $t_{\text{mix}}(0.001)$, after which the simulation begins.  Results are shown in Table \ref{simulation}.  The estimated variances are then used to evaluate the standard errors of the estimates of the means.  Table \ref{exact,sim} reports these results, together with exact computations of the mean for small $M$ and $N$.
     
\begin{table}[htb]
\caption{\label{simulation}Exact ($\mu$, $\sigma^2$) and simulated ($\hat\mu$, $\hat\sigma^2$) means and variances of profit per turn at game $B$, game $(A+B)/2$, and game $A^2B^2$ with $(p_0,p_1,p_2,p_3,p_4)=(1/20,3/20,8/13,3/4,9/10)$. The third column $l$ is the warming-up period and the forth column $c$ is the constant for the block size $b:=\lfloor cn^{1/3}\rfloor$.}
\catcode`@=\active \def@{\hphantom{0}}
\begin{center}
\begin{footnotesize}
\begin{tabular}{cccccccc}
\noalign{\smallskip}
\hline
\noalign{\smallskip}
$(M,N)$  & $n$  & $l$ & $c$ & $\mu_B$   &   $\hat{\mu}_B$  & $\sigma^2_B$  & $\hat{\sigma}^2_B$ \\
\noalign{\smallskip}
\hline
\noalign{\smallskip}
$(3,3)$     & $10^8$    & $10^2$        & $10$ & $-0.209606$ & $-0.209695@$ & 113.864 & 111.544 \\
$(4,4)$     & $10^8$    & $5\times10^2$ & $10$ & $-0.188909$ & $-0.189613@$ & 228.548 & 213.019 \\
$(5,5)$     & $10^8$    & $10^3$        & $10$ &             & $-0.143901@$ &         & 306.867 \\
$(10,10)$   & $10^9$    & $10^4$        & $10$ &             & $-0.0568051$ &         & 456.483 \\
$(20,20)$   & $10^9$    & $10^4$        & $20$ &             & $-0.0464865$ &         & 355.961 \\
$(50,50)$   & $10^9$ & $10^5$        & $50$ &             & $-0.0472350$ &         & 197.114  \\
$(100,100)$ & $10^9$ & $10^6$        & $100$ &             &$-0.0484520$ &         & 120.574  \\
\noalign{\smallskip}
\hline
\end{tabular}
\end{footnotesize}
\end{center}

\begin{center}
\begin{footnotesize}
\begin{tabular}{cccccccc}
\noalign{\smallskip}
\hline
\noalign{\smallskip}
$(M,N)$  & $n$  & $l$ & $c$ & $\mu_{(1/2,1/2)}$   &   $\hat{\mu}_{(1/2,1/2)}$  & $\sigma^2_{(1/2,1/2)}$  & $\hat{\sigma}^2_{(1/2,1/2)}$ \\
\noalign{\smallskip}
\hline
\noalign{\smallskip}
$(3,3)$     & $10^8$    & $10^2$   & $10$ &0.0162586 & 0.0165893 & 4.59703 & 4.52893 \\
$(4,4)$     & $10^8$    & $5\times10^2$ & $10$ &  0.0229270 & 0.0227338 & 4.79211 & 4.74619 \\
$(5,5)$     & $10^8$    & $10^3$        & $10$ &   & 0.0242763 &         & 4.75117 \\
$(10,10)$   & $10^9$    & $10^4$        & $10$ &  & 0.0243652 &         & 4.77696 \\
$(20,20)$   & $10^9$    & $10^4$        & $20$ &  & 0.0245289  &         &   4.69812\\
$(50,50)$   & $10^9$    & $10^5$        & $50$ &   & 0.0246507  &         &  4.44180\\
$(100,100)$ & $10^9$    & $10^6$        & $100$ &   & 0.0244565 &         &         4.13757\\
\noalign{\smallskip}
\hline
\end{tabular}
\end{footnotesize}
\end{center}

\begin{center}
\begin{footnotesize}
\begin{tabular}{cccccccc}
\noalign{\smallskip}
\hline
\noalign{\smallskip}
$(M,N)$  & $n$  & $l$ & $c$ &  $\mu_{[2,2]}$   &   $\hat{\mu}_{[2,2]}$  & $\sigma^2_{[2,2]}$  & $\hat{\sigma}^2_{[2,2]}$\\
\noalign{\smallskip}
\hline
\noalign{\smallskip}
$(3,3)$     & $10^8$    & $10^2$   & $10$ &0.0172959 & 0.0175025 & 4.26540 & 4.26399\\
$(4,4)$     & $10^8$    & $5\times10^2$ & $10$ &  0.0231048 & 0.0231664 &         & 4.52584\\
$(5,5)$     & $10^8$    & $10^3$        & $10$ &     & 0.0241232 &         & 4.63878\\
$(10,10)$   & $10^9$    & $10^4$        & $10$ &   & 0.0244755 &         & 4.66672\\
$(20,20)$   & $10^9$    & $10^4$        & $20$ &   & 0.0244955 &         & 4.65870\\
$(50,50)$   & $10^9$ & $10^5$        & $50$ &   &  0.0244899 &         & 4.39306 \\
$(100,100)$   & $10^9$ & $10^6$        & $100$ &   & 0.0244883 &         & 4.07808 \\
\noalign{\smallskip}
\hline
\end{tabular}
\end{footnotesize}
\end{center}
\end{table}

\begin{table}[htb]
\caption{\label{exact,sim}Exact and simulated mean profits from game $B$, game $(1/2)(A+B)$, and game $A^2 B^2$ with $\bm p=(p_0,p_1,p_2,p_3,p_4)=(1/20,3/20,8/13,3/4,9/10)$.  Results are rounded to six significant digits.  For $MN\le20$, results are exact.  For $MN>20$, results are from simulations (see Table \ref{simulation} for more information) and standard errors are provided.\medskip}
\catcode`@=\active \def@{\hphantom{0}}
\catcode`#=\active \def#{\hphantom{$-$}}
\tabcolsep=1.5mm
\begin{center}
\begin{footnotesize}
\begin{tabular}{ccccccc}
\noalign{\smallskip}
\hline
\noalign{\smallskip}
$(M,N)$ & $\mu_B$ & (st.~error) & $\mu_{(1/2,1/2)}$ & (st.~error) & $\mu_{[2,2]}$ & (st.~error)\\
\noalign{\smallskip}
\hline
\noalign{\smallskip}
$(3,3)$ & $-0.209606@$ &---& 0.0162586   &---& 0.0172959 &---\\ % OK
$(3,4)$ & $-0.218065@$ &---& 0.0187059   &---& 0.0195027 &---\\ % OK
$(3,4)$ & $-0.220219@$ &---& 0.0190801   &---& 0.0197024 &---\\ % OK
$(3,6)$ & $-0.221078@$ &---& 0.0191405   &---& 0.0196551 &---\\ % OK
\noalign{\medskip}
$(4,4)$ & $-0.188909@$ &---& 0.0229270   &---& 0.0231048 &---\\ % OK
$(4,5)$ & $-0.171680@$ &---& 0.0235580   &---&   \\
\noalign{\smallskip}
\hline
\noalign{\smallskip}
$(5,5)$ & $-0.143901@$ & $(0.00175176)@$ & 0.0242763 & $(0.000217972)@$ & 0.0241232 & $(0.000215378)@$\\
$(10,10)$ & $-0.0568051$ & $(0.000675635)$ & 0.0243652 & $(0.0000691156)$ & 0.0244755 & $(0.0000683134)$ \\
$(20,20)$ & $-0.0464865$ & $(0.000596625)$ & 0.0245289 & $(0.0000685428)$ & 0.0244955 & $(0.0000682547)$ \\
$(50,50)$ & $-0.0472350$ & $(0.000443975)$ & 0.0246507 & $(0.0000666468)$ & 0.0244899 & $(0.0000662802)$ \\
$(100,100)$ & $-0.0484520$ & $(0.000347238)$ & 0.0244565 & $(0.0000643239)$ & 0.0244883 & $(0.0000638598)$ \\
\noalign{\smallskip}
\hline`
\end{tabular}
\end{footnotesize}
\end{center}
\end{table}

Thus, the length of a simulation of the mean profit will depend on both the rate of convergence to equilibrium of the underlying Markov chain and on the variance of the profit, which determines the standard error.  Mihailovi\'c and Rajkovi\'c \cite {MR06} did not realize this and simulated with a much too small sample size.  This led them to the conclusion that ``capital evolution depends to a large degree on the lattice [i.e., array] size.''  As we have seen, there is some evidence that mean profit converges as $M,N\to\infty$, but a more persuasive argument against this conclusion can be given by rerunning the simulations that led to this conclusion with an adequate sample size.  The specific parameter values used in Ref.~\cite{MR06} were not revealed, making it impossible to replicate the experiment.  However, we can use the default parameters in one of the programs of Mihailovi\'c and Rajkovi\'c, namely $p_0=1/20$, $p_1=3/20$, $0\le p_2\le 1$, $p_3=3/4$, and $p_4=9/10$.  See Figure~\ref{MR}.

\begin{figure}[htb]
\centering
\includegraphics[width = 3.in]{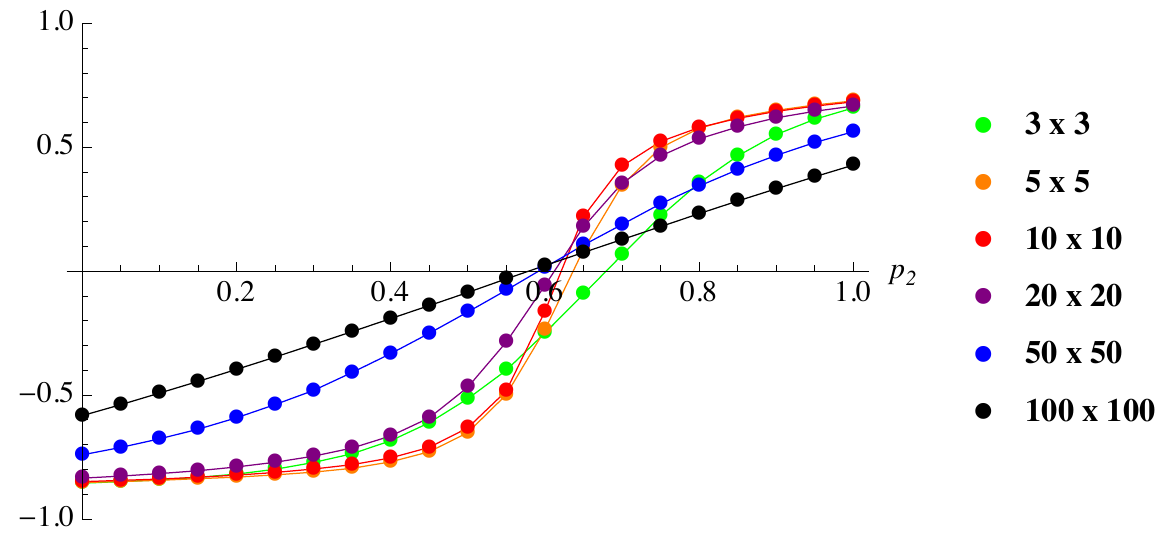}
\includegraphics[width = 3.in]{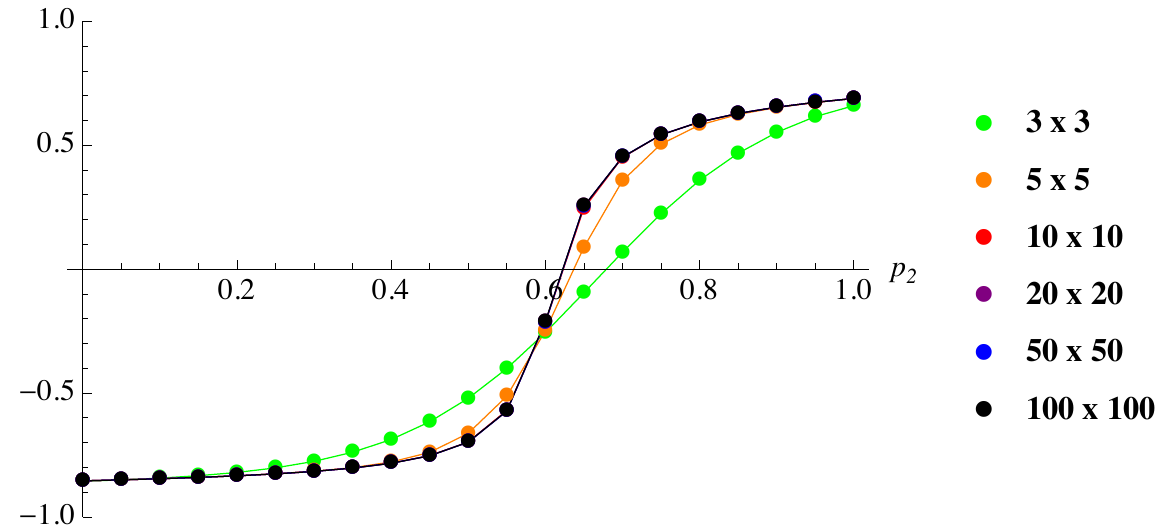}
\caption{\label{MR}In both figures the simulated means $\mu_B(1/20,3/20,p_2,3/4,9/10)$ are graphed as functions of $p_2$.  The distinction is that, in the first figure, we used the sample size of Mihailovi\'c and Rajkovi\'c \cite {MR06};  in the second figure, we used a larger sample size that is suitable for this problem.  In the latter figure, the curve for the $100\times100$ case hides the curves for the $10\times10$, $20\times20$, and $50\times50$ cases.}
\end{figure}

\section{Conclusions}

We considered the spatially dependent Parrondo games of Mihailovi\'c and Raj\-kovi\'c \cite{MR06}, which assume (a) $MN$ players arranged in an $M\times N$ array with periodic boundary conditions and (b) five possibly biased coins, numbered 0--4.  In game $B$ a randomly chosen player tosses coin $m$, where $m$ is the number of winners among the player's four nearest neighbors.  Game $A$ is the special case in which all five coins are fair, and game $C$ combines game $A$ and game $B$.  More precisely, game $C$ is either the randomly mixed game $C:=\gamma A+(1-\gamma)B$, where $0<\gamma<1$, or the nonrandom periodic pattern $C:=A^rB^s$, where $r,s\ge 1$.  We obtained a strong law of large numbers and a central limit theorem for the sequence of profits (1 for a win, $-1$ for a loss) to the set of $MN$ players, assuming repeated play of game $B$ or of game $C$.  To maximize the values of $M$ and $N$ for which exact computations are feasible, we regarded states as equivalent if they are equal after rotation and/or reflection of rows and/or columns of the $M\times N$ array of players, as well as after matrix transposition if $M=N$.  This allowed us to compute $\mu_B$ (and $\mu_C$), the mean profit to the set of $MN$ players playing game $B$ (or game $C$) for pairs $(M,N)$ with $3\le M\le N$ and $MN\le20$.  When $(M,N)=(3,3)$ or $(M,N)=(3,4)$ we can graph three-dimensional cross-sections of the Parrondo and anti-Parrondo regions for specified $p_0$ and $p_4$.  When $(M,N)=(3,5)$ or $(M,N)=(4,4)$ we can graph two-dimensional cross-sections of the Parrondo and anti-Parrondo regions for specified $p_0$, $p_1$, and $p_4$.  We found sufficient conditions for the means $\mu_B$ and $\mu_C$ (with $C:=(A+B)/2$) to converge as $M,N\to\infty$.  Together with simulation results of an adequate sample size, this casts doubt on a finding of Mihailovi\'c and Rajkovi\'c that ``capital evolution depends to a large degree on the lattice size.''  Finally, we found that monotonicity of the mean profit function $\mu_B$, as a function of $(p_0,p_1,p_2,p_3,p_4)$, where $p_m$ is the probability of heads for coin $m$, does not hold in general.   It does hold on the subset $0\le p_0\le p_1\le p_2\le p_3\le p_4\le1$.

\end{document}